\documentclass[reqno]{amsart}

\usepackage{amsmath}
\usepackage{amsfonts}
\usepackage{amssymb,enumerate}
\usepackage{amsthm}
\usepackage[all]{xy}
\usepackage{rotating}
\usepackage{hyperref}
\usepackage{color}

\theoremstyle{plain}
\newtheorem{lem}{Lemma}[section]
\newtheorem{cor}[lem]{Corollary}
\newtheorem*{mainthm*}{Main Theorem}
\newtheorem{prop}[lem]{Proposition}

\newtheorem{intthm}{Theorem}

\theoremstyle{definition}

\newtheorem{ex}[lem]{Example}

\newtheorem{disc}[lem]{Remark}

\newtheorem*{convention*}{Convention}

\newcommand{\id}{\operatorname{id}}

\newcommand{\depth}{\operatorname{depth}}

\newcommand{\soc}{\operatorname{Soc}}

\newcommand{\coker}{\operatorname{Coker}}

\newcommand{\im}{\operatorname{Im}}

\newcommand{\End}{\operatorname{End}}

\newcommand{\Ker}{\operatorname{Ker}}

\newcommand{\ideal}[1]{\mathfrak{#1}}
\newcommand{\m}{\ideal{m}}

\newcommand{\fm}{\ideal{m}}

\newcommand{\xra}{\xrightarrow}

\renewcommand{\geq}{\geqslant}
\renewcommand{\leq}{\leqslant}
\renewcommand{\ker}{\Ker}

\newcommand{\Hom}{\operatorname{Hom}}

\def\soc{\operatorname{Soc}}

\numberwithin{equation}{lem}

\begin{document}

\bibliographystyle{amsplain}

\title[Irreducible homomorphisms to/from free and injective modules]{Structure of irreducible homomorphisms\\ to/from free and injective modules}

\author{Saeed Nasseh}
\address{Department of Mathematical Sciences\\
Georgia Southern University\\
Statesboro, GA 30460, U.S.A.}
\email{snasseh@georgiasouthern.edu}
\urladdr{https://cosm.georgiasouthern.edu/math/saeed.nasseh}





\begin{abstract}
Let $R$ be a commutative noetherian local ring. We extend the work of the author and Takahashi~\cite{NT} by investigating the structure of irreducible monomorphisms originating from free modules in the category of finitely generated $R$-modules. In the case where $R$ is complete, we further study the structure of irreducible monomorphisms and epimorphisms to and from injective modules in the category of artinian $R$-modules.
\end{abstract}

\maketitle


\section{Introduction}

\begin{convention*}
Throughout the paper, $(R,\fm,k)$ is a commutative noetherian local ring. When we say $R$ is complete, we mean it is complete in the $\fm$-adic topology. We write $E=E_R(k)$ for the injective envelope of $k$. For an $R$-module $M$ and a positive integer $n$, by $M^{\oplus n}$ we denote the direct sum $\bigoplus_{i=1}^nM$ and $M^{\vee}=\Hom_R(M,E)$ is the Matlis dual of $M$. Moreover, $\mu_R(M)$ and $\soc_R(M)$ denote the minimal number of generators of $M$ and the socle of $M$, respectively. We simply write $\soc(R)$ for $\soc_R(R)$. Finally, the quotient $E/\soc_R(E)$ will be denoted by $\overline{E}$.
\end{convention*}

An $R$-module homomorphism $f\colon M\to N$ is called \emph{irreducible} if $f$ is neither a split monomorphism nor a split epimorphism, and for every factorization
$$
\xymatrix{
&L\ar[rd]^{h}&\\
M\ar[ru]^{g}\ar[rr]^{f}&&N
}
$$
of $f$ we have $g$ is a split monomorphism or $h$ is a split epimorphism; references on irreducible homomorphisms and related topics include~\cite{ARS, roger, yoshino}.\vspace{2mm}

The structure of irreducible monomorphisms to and irreducible epimorphisms from free modules in the category of finitely generated $R$-modules has been completely determined by Nasseh and Takahashi~\cite[Theorems A and B]{NT}. On the other hand, irreducible epimorphisms onto free $R$-modules cannot exist since any such epimorphism necessarily splits. In this paper, we complete this picture by studying the structure of irreducible monomorphisms from free modules in the category of finitely generated $R$-modules, which constitutes the only remaining case. Our main result in this direction is the following theorem, whose proof is given in Section~\ref{sec05062016a}.

\begin{intthm}\label{thm20260913a}
Let $(R,xR)$ be a discrete valuation ring and $\phi\colon R\to M$ be an irreducible monomorphism in the category of finitely generated $R$-modules. Then, there exists a commutative diagram
\begin{equation}\label{eq20260913c}
\xymatrix{
R\ar[rr]^{\phi}\ar@{=}[d]&&M\ar[d]^-{\theta}\\
R\ar[rr]^-{\varpi}&&R
}
\end{equation}
in which $\theta$ is an isomorphism and $\varpi$ is the multiplication by $x$.
\end{intthm}

In Section~\ref{sec20260919a}, using Matlis duality extensively, we establish the dual versions of the results in~\cite{NT} and Theorem~\ref{thm20260913a}. We first turn our attention to the case where we have an irreducible epimorphism from an injective module in the category of artinian $R$-modules, obtaining the following result.

\begin{intthm}\label{thm20260919a}
Assume that $R$ is complete. Then, the following assertions hold.
\begin{enumerate}[\rm(a)]
\item
Let $J'$ be a non-zero $R$-module that is a proper direct summand of $\overline{E}$. Then, the composition $\delta'\colon E\xra{\theta'} \overline{E}\to J'$ of the natural surjections is an irreducible epimorphism in the category of artinian $R$-modules.
\item
Let $\psi\colon I\to M$ be an irreducible epimorphism in the category of artinian $R$-modules with $I$ injective and $\soc_R(I)\subseteq \ker(\psi)$. Then, the following hold.
\begin{enumerate}[\rm(b1)]
\item
The $R$-module $M$ is isomorphic to a direct summand of $\overline{E}$.
\item
If $M$ is indecomposable, then for every monomorphism $\iota\colon E\to I$ there exists a split epimorphism $\eta'\colon \overline{E}\to M$ such that the diagram
\begin{equation}\label{eq20260915d}
\xymatrix{
E\ar[r]^{\theta'}\ar[d]_-{\iota}&\overline{E}\ar[d]^-{\eta'}\\
I\ar[r]^{\psi}&M
}
\end{equation}
commutes in the category of artinian $R$-modules.
\end{enumerate}
\end{enumerate}  
\end{intthm}

The next result considers the case where we have an irreducible monomorphism to an injective module in the category of artinian $R$-modules.

\begin{intthm}\label{thm20260917a}
Assume that $R$ is complete and let $\psi\colon M\to I$ be an irreducible monomorphism in the category of artinian $R$-modules with $I\cong E^{\oplus n}$ for some positive integer $n$. Then, the following assertions hold.
\begin{enumerate}[\rm(a)]
\item
The cokernel of $\psi$ is isomorphic to $k$.
\item
Assume that $\End_R(M)$ is a local ring, and let $\pi\colon I\to E$ be a split epimorphism. Then, the following hold.
\begin{enumerate}[\rm(b1)]
\item
The composition $\pi \psi\colon M\to E$ is irreducible in the category of artinian $R$-modules and hence, it is either surjective or injective.
\item
If $\pi \psi$ is injective, then $\dim_k(\soc(R))=1$, $n=1$ (that is, there exists an isomorphism $\beta\colon I\to E$), and there is a commutative diagram
\begin{equation}\label{eq20260917a}
\xymatrix{
M\ar[r]^{\psi}\ar[d]_-{\alpha}&I\ar[d]^-{\beta}\\
N\ar[r]^{\varsigma}&E
}
\end{equation}
in which $N$ is a maximal submodule of $E$ (that is, $E/N\cong k$), the map $\varsigma$ is the natural inclusion, and $\alpha$ is an isomorphism.
\end{enumerate}
\end{enumerate}
\end{intthm}

Finally, we establish the following result concerning the irreducible epimorphisms to injective modules in the category of artinian $R$-modules. Note that irreducible monomorphisms from injective $R$-modules cannot exist since any such monomorphism necessarily splits.

\begin{intthm}\label{thm20260917c}
Let $(R,xR)$ be a complete discrete valuation ring and $\psi\colon M\to E$ be an irreducible epimorphism in the category of artinian $R$-modules. Then, there exists a commutative diagram
\begin{equation}\label{eq20260917e}
\xymatrix{
M\ar[rr]^{\psi}\ar[d]_-{\theta}&&E\ar@{=}[d]\\
E\ar[rr]^-{\varpi'}&&E
}
\end{equation}
in which $\theta$ is an isomorphism and $\varpi'$ is the multiplication by $x$.
\end{intthm}



\section{Irreducible monomorphisms from free modules}\label{sec05062016a}

In this section, we give the proof of Theorem~\ref{thm20260913a}. We start by providing an example that demonstrates the assumptions and conclusion of this theorem.

\begin{ex}\label{ex20260919a}
Let $k$ be a field and $R=k[\![x]\!]$ and consider the monomorphism $\phi\colon R\to R$ given by the multiplication by $x$. As $x$ is not a unit element in $R$, the map $\phi$ is neither a split monomorphism nor a split epimorphism. Now, let $R\xra{g}L\xra{h}R$ be a factorization of $\phi$. Note that $g$ is a monomorphism. If $\im(h)=R$, then $h$ is a split epimorphism. However, if $\im(h)\subseteq (x)$, then for every $\ell\in L$, there exists a unique element $r\in R$ such that $h(\ell)=xr$. Defining the $R$-module homomorphism $h'\colon L\to R$ by $h'(\ell)=r$, one can check that $h'g=\id_R$, that is, $g$ is a split monomorphism. Therefore, $\phi$ is an irreducible monomorphism. This matches Theorem~\ref{thm20260913a} as every irreducible monomorphism $R\to M$ is isomorphic to such a map $\phi$.
\end{ex}

The notation $\phi_N$ introduced in the following remark will be used throughout the subsequent results in this section without further comments.

\begin{disc}\label{disc20260911a}
Let $\phi\colon L\to M$ be a monomorphism of $R$-modules. Note that if $\pi\colon M\to \coker(\phi)$ is the natural surjection and $N$ is a submodule of $\coker(\phi)$, then $\im(\phi)\subseteq \pi^{-1}(N)$, that is, there is an induced inclusion map $\phi_N\colon L\to \pi^{-1}(N)$. If, in addition, we assume that $\phi$ is irreducible, then $\coker(\phi)$ is a non-zero $R$-module since $\phi$ is not an isomorphism.
\end{disc}

\begin{prop}\label{lem20260911a}
Let $\phi\colon F\to M$ be an irreducible monomorphism of $R$-modules with $F$ free and $\pi\colon M\to \coker(\phi)$ be the natural surjection. Then, for every proper submodule $N$ of $\coker(\phi)$, the inclusion map $\phi_N\colon F\to \pi^{-1}(N)$ is a split monomorphism and we have an isomorphism
\begin{equation}\label{eq20260911a}
\pi^{-1}(N)\cong F\oplus N.
\end{equation}
\end{prop}

\begin{proof}
By our discussion in Remark~\ref{disc20260911a}, note that $\im(\phi)\subseteq \pi^{-1}(N)\subseteq M$. We then have a factorization $F\xra{\phi_N} \pi^{-1}(N)\xra{h} M$ of $\phi$, where $h$ is the natural inclusion. If $h$ is a split epimorphism, then it is an isomorphism and thus, $\pi^{-1}(N)=M$. This implies that $N=\coker(\phi)$, which is a contradiction. Therefore, by our assumption we conclude that $\phi_N$ is a split monomorphism.
Moreover, we have a split short exact sequence
$0\to F\xra{\phi_N} \pi^{-1}(N)\xra{\pi} N\to 0$ which implies the isomorphism~\eqref{eq20260911a}.
\end{proof}

\begin{cor}\label{thm20260910a}
Let $\phi\colon F\to M$ be an irreducible monomorphism in the category of finitely generated $R$-modules with $F$ free and $\pi\colon M\to \coker(\phi)$ be the natural surjection. Then, for every positive integer $n$ we have an isomorphism
$$
F\oplus \fm^n \coker(\phi)\cong \pi^{-1}(\fm^n \coker(\phi)).
$$
\end{cor}

\begin{proof}
By Nakayama's Lemma, $\fm^n \coker(\phi)$ is a proper submodule of $\coker(\phi)$. The assertion now follows from Proposition~\ref{lem20260911a} by setting $N=\fm^n \coker(\phi)$. 
\end{proof}

Proposition~\ref{lem20260911a} does not provide a complete classification of the type obtained in~\cite[Theorems A and B]{NT}. Instead, Theorem~\ref{thm20260913a} provides a satisfactory classification when $R$ is a discrete valuation ring. Before proving this theorem, we record the following remark. Again, the notation introduced in this remark will be used freely in what follows.

\begin{disc}\label{disc20260912a}
Let $(R,xR)$ be a discrete valuation ring and $M$ be a finitely generated $R$-module. By the structure theorem for finitely generated modules over a principal ideal domain, $M$ has a decomposition
\begin{equation}\label{eq20260919a}
M\cong F\oplus \tau(M)
\end{equation}
where $F$ is a finitely generated free $R$-module and $\tau(M)$ denotes the torsion submodule of $M$. If $\tau(M)\neq 0$, then there exist positive integers $s, a_1,\ldots,a_s$ with
\begin{equation}\label{eq20260919b}
\tau(M)\cong \bigoplus_{i=1}^s\left(R/(x^{a_i})\right).
\end{equation}
Note that $F\cong M/\tau(M)$ and using the isomorphism~\eqref{eq20260919b}, for every element $\ell\in \tau(M)$ there exists a positive integer $n_{\ell}$ such that $x^{n_{\ell}}\ell=0$.
\end{disc}

\noindent \emph{Proof of Theorem~\ref{thm20260913a}.}
Let $m=\phi(1)$ and note that $\phi(r)=rm$ for all elements $r\in R$. We identify $M$ by $F\oplus \tau(M)$ under the isomorphism~\eqref{eq20260919a}, where $F=R^{\oplus n}$ for a non-negative integer $n$, and write $m=(f,\ell)$ with $f\in F$ and $\ell\in \tau(M)$. Let $\pi\colon M\to \coker(\phi)$ be the natural surjection. The proof of Theorem~\ref{thm20260913a} is now given step by step as follows.
\vspace{2mm}

Step 1: If $f=0$, then $m=(0,\ell)$. In particular, $\ell\neq 0$ as $\phi$ is not the zero map and thus, there is a positive integer $n_{\ell}$ such that $x^{n_{\ell}}m=0$. This contradicts the fact that $\phi$ is a monomorphism. Hence, $f\neq 0$ and therefore, $n\geq 1$.
\vspace{2mm}
 
Step 2: The element $f$ is not primitive (i.e., it cannot be extended to a basis for the free $R$-module $F$). To see this, assume on the contrary that it is primitive. Then, $Rf$ is a free direct summand of $F$ and considering the natural surjection $\psi\colon M\to Rf$, the composition $\psi\phi\colon R\to Rf$ has the property that $\psi\phi(1)=f$. Identifying $Rf$ under the isomorphism $Rf\cong R$, the composition $\psi\phi\colon R\to R$ is then an isomorphism, which contradicts the fact that $\phi$ is irreducible.
\vspace{2mm}

Step 3: Suppose that $n\geq 2$. Since, by Step 2, $f$ is not primitive, we have $f=x^df'$, where $f'\in F$ is primitive and $d$ is a positive integer. Moreover, we can write $F=Rf'\oplus R^{\oplus(n-1)}$ and $m=x^d(f',0)+(0,\ell)$ is an element in $N'=Rf'\oplus \tau(M)$. Hence, $\im(\phi)\subseteq N'$. Note that $N'$ is a proper submodule of $M$ and thus, $N=\pi(N')$ is a proper submodule of $\coker(\phi)$. Then, by Proposition~\ref{lem20260911a}, the inclusion map $\phi_N\colon R\to N'$ is a split monomorphism. Hence, there is an $R$-module homomorphism $\rho\colon N'\to R$ such that $\rho\phi_N=\id_R$. As $M=N'\oplus R^{\oplus(n-1)}$, the natural inclusion $\varepsilon\colon N'\to M$ is a split monomorphism with $\lambda \varepsilon=\id_{N'}$, where $\lambda\colon M\to N'$ is the natural surjection. Note that $\rho \lambda\colon M\to R$ and we have
$$
\rho \lambda \phi=\rho \lambda \varepsilon \phi_N=\rho \id_{N'} \phi_N= \rho \phi_N=\id_R.
$$
This means that $\phi$ is a split monomorphism, which is a contradiction. Therefore, our assumption that $n\geq 2$ is not valid and hence, it follows from Step 1 that $n=1$. Moreover, in this situation, we can rewrite $m$ as $m=(x^d,\ell)$ with $d\geq 1$.
\vspace{2mm}

Step 4: Using $M=R\oplus \tau(M)$, write $\phi=\begin{pmatrix}\phi_1& \phi_2\end{pmatrix}^{tr}$ with $\phi_1\colon R\to R$ and $\phi_2\colon R\to \tau(M)$ defined by $\phi_1(r)=rx^d$ and $\phi_2(r)=r\ell$ for all elements $r\in R$. It follows from~\cite[Lemma 2.3(b)]{NT} that $\phi_1$ is irreducible. If $d\geq 2$, then $R\xra{g}R\xra{h} R$ is a factorization of $\phi_1$, where $g$ is the multiplication map by $x^{d-1}$ and $h$ is the multiplication map by $x$. However, none of these maps can be split because $x$ is not unit. This contradiction implies that $d=1$, and hence, $m=(x,\ell)$.
\vspace{2mm}

Step 5: Assume that $\tau(M)\neq 0$. Let $\ell=x\ell'$ for some $\ell'\in \tau(M)$. Then, $m=xu$, where $u=(1,\ell')$, and we have the factorization $R\xra{g'}R\xra{h'}M$ of $\phi$ in which $g'$ is the multiplication map by $x$ and $h'$ is defined by $h'(r)=ru$ for all $r\in R$. Since $g'$ is not a split monomorphism and $\phi$ is irreducible, $h'$ must be a split epimorphism. Note that $h'$ is a monomorphism as well and hence, it is an isomorphism, that is, $M\cong R$. This is impossible as we assume that $\tau(M)\neq 0$. Therefore, if $\tau(M)\neq 0$, then $\ell\notin x\tau(M)$. This means that $\ell$ is part of some minimal generating set of $\tau(M)$.\vspace{2mm}

Step 6: Suppose $\mu_R(\tau(M))\geq 2$. It follows from Step 5 that $\tau(M)=R\ell\oplus X$, where $X$ is a non-zero submodule of $\tau(M)$. In this situation, $L'=R\oplus R\ell$ is a proper submodule of $M$ and we have $m=(x,\ell)\in L'$. Hence, $\im(\phi)\subseteq L'$ and $L=\pi(L')$ is a proper submodule of $\coker(\phi)$. By an argument analogous to that in Step 3 we conclude that $\phi$ is a split monomorphism, which is a contradiction. Thus, $\mu_R(\tau(M))\leq 1$.\vspace{2mm}

Step 7: If $\tau(M)\neq 0$, then by Step 6 we have $\mu_R(\tau(M))=1$, that is, $\tau(M)$ is a cyclic $R$-module. Note that by Step 5, the element $\ell$ is a generator for $\tau(M)$. Therefore, we obtain an isomorphism
$$
\tau(M)\cong R/(x^c)
$$
for a positive integer $c$ under which $\ell\in \tau(M)$ corresponds to $\overline{1}\in R/(x^c)$. This yields the isomorphism $\theta\colon M\xra{\cong} R\oplus R/(x^c)$ defined by $\theta(x,\ell)=(x,\overline{1})$ and consequently, we obtain a commutative diagram
\begin{equation}\label{eq20260930a}
\xymatrix{
R\ar[rr]^{\phi}\ar@{=}[d]&&M\ar[d]^-{\theta}\\
R\ar[rr]^-{\varpi_c}&&R\oplus R/(x^c)
}
\end{equation}
in which for every element $r\in R$ we have $\varpi_c(r)=(rx,\overline{r})$. Note that $\varpi_c$ is irreducible because $\phi$ is irreducible.\vspace{2mm}

Step 8: If $\tau(M)\neq 0$, then setting $L=R\oplus R/(x^{c+1})$, the map $\varpi_c$ from Step 7 has a factorization $R\xra{g} L\xra{h} R\oplus R/(x^c)$, where for any element $r\in R$ we have $g(r)=(rx,\overline{r}_{c+1})$ with $\overline{r}_{c+1}$ denoting $r$ mod $(x^{c+1})$ and $h=\id_R\oplus\pi_c$ with $\pi_c\colon R/(x^{c+1})\to R/(x^c)$ being the natural surjection. Note that $g$ is not a split monomorphism. In fact, any $\varrho\colon L\to R$ vanishes on the torsion summand and thus, $\varrho(g(1))=\varrho(x,\overline{1}_{c+1})=\varrho(x,\overline{0}_{c+1})=x\varrho(1,\overline{0}_{c+1})\in(x)$. Therefore, $\varrho g\neq \id_R$. Also, $h$ is not a split epimorphism. In fact, any map $\zeta\colon R/(x^{c})\to R/(x^{c+1})$ with $\zeta\pi_c=\id_{R/(x^{c})}$ satisfies $x^c\zeta(\overline{1})=0$, that is, $\zeta(\overline{1})\in(x)/(x^{c+1})$. However, then we have $\pi_c(\zeta(\overline{1}))\in(x)/(x^c)$ which means $\pi_c(\zeta(\overline{1}))\neq \overline{1}$, a contradiction.

Thus, neither $g$ nor $h$ splits, contradicting irreducibility of $\varpi_c$. Hence, $\tau(M)=0$ and consequently, we obtain the commutative diagram~\eqref{eq20260913c} from~\eqref{eq20260930a}.
\qed

\section{Irreducible homomorphisms to/from injective modules}\label{sec20260919a}

This section is entirely devoted to the proofs of Theorems~\ref{thm20260919a}, \ref{thm20260917a}, and~\ref{thm20260917c}. The main tool used throughout the proofs is Matlis duality, for which we provide some preliminary background before proceeding to the arguments; see~\cite{bruns, M}.

\begin{disc}\label{para20260915c}
Let $M$ be an $R$-module. Then, the canonical $R$-module homomorphism $\chi_M\colon M\to M^{\vee\vee}$ is injective. Also, if $f\colon M\to N$ is an $R$-module homomorphism, then the diagram
\begin{equation}\label{eq20260915a}
\xymatrix{
M\ar[rr]^f\ar[d]_-{\chi_M}&&N\ar[d]^-{\chi_N}\\
M^{\vee\vee}\ar[rr]^{f^{\vee\vee}}&&N^{\vee\vee}
}
\end{equation}
is commutative. Recall that the $R$-module $M$ is called \emph{Matlis reflexive} if $\chi_M$ is an isomorphism, that is, $M^{\vee\vee}\cong M$. Note that if $R$ is complete, then all finitely generated and artinian $R$-modules are Matlis reflexive.
\end{disc}

\begin{disc}\label{para20260915b}
If $R$ is complete, then the Matlis dual functor $(-)^{\vee}$ is a contravariant equivalence between the category of finitely generated and that of artinian $R$-modules. Moreover, this functor is full and faithful, that is, it yields a bijection
$$
\Phi_{M,N}\colon \Hom_R(M,N)\to \Hom_R(N^{\vee},M^{\vee})
$$
defined by $\Phi_{M,N}(f)=f^{\vee}$ between the Hom-set of the finitely generated $R$-modules $M, N$ and the Hom-set of their artinian duals $M^{\vee}, N^{\vee}$.
\end{disc}

\begin{disc}\label{para20260914b}
Let $M$ be a Matlis reflexive $R$-module. Then, $M$ is indecomposable if and only if so is $M^{\vee}$. Moreover, the dualization induces an anti-isomorphism $\End_R(M)\cong \End_R(M^{\vee})^{op}$. Therefore, $\End_R(M)$ is a local ring if and only if so is $\End_R(M^{\vee})$.
\end{disc}

Using the basic categorical properties of the Matlis dual functor $(-)^{\vee}$ and diagram~\eqref{eq20260915a}, it is straightforward to check the following lemma.

\begin{lem}\label{lem20260915a}
Let $f\colon M\to N$ be an $R$-module homomorphism. Then, the following assertions hold.
\begin{enumerate}[\rm(a)]
\item
If $f$ is a split monomorphism (resp. epimorphism), then $f^{\vee}$ is a split epimorphism (resp. monomorphism).
\item
If $M$ and $N$ are Matlis reflexive and $f^{\vee}$ is a split monomorphism (resp. epimorphism), then $f$ is a split epimorphism (resp. monomorphism).
\end{enumerate}
\end{lem}

\begin{prop}\label{lem20260914a}
Assume that $R$ is complete, and let $f\colon M\to N$ be homomorphism of finitely generated (resp. artinian) $R$-modules. Then, $f$ is irreducible in the category of finitely generated (resp. artinian) $R$-modules if and only if $f^{\vee}$ is irreducible in the category of artinian (resp. finitely generated) $R$-modules.
\end{prop}

\begin{proof}
Assume that $f$ is irreducible in the category of finitely generated $R$-modules. It follows from Lemma~\ref{lem20260915a}(b) that $f^{\vee}$ is neither a split monomorphism nor a split epimorphism. Now, let $N^{\vee}\xra{h'}L\xra{g'}M^{\vee}$ be a factorization of $f^{\vee}$ through an artinian $R$-module $L$. Note that $L^{\vee}$ is a finitely generated $R$-module and by our discussion in Remark~\ref{para20260915b} there exist unique $R$-module homomorphisms $g\colon M\to L^{\vee}$ and $h\colon L^{\vee}\to N$ with $g'=g^{\vee}$ and $h'=h^{\vee}$ such that $M\xra{g}L^{\vee}\xra{h}N$ is a factorization of $f$. By our assumption, $g$ is a split monomorphism or $h$ is a split epimorphism. Hence, by Lemma~\ref{lem20260915a}(a) we conclude that $h'$ is a split monomorphism or $g'$ is a split epimorphism. Thus, $f^{\vee}$ is irreducible in the category of artinian $R$-modules, as desired.

The converse of this proposition follows by a similar argument.

The case where the $R$-module homomorphism $f\colon M\to N$ is irreducible in the category of artinian $R$-modules if and only if $f^{\vee}$ is irreducible in the category of finitely generated $R$-modules is treated similarly.
\end{proof}

\begin{disc}
Note that $\soc_R(E)\cong k$ and we have the isomorphisms
\begin{equation}\label{eq20260915b}
\fm^{\vee}\cong \overline{E}\cong E/k.
\end{equation}
\end{disc}

\begin{lem}\label{lem20260915b}
Let $f\colon M\to N$ be an $R$-module homomorphism. Then, $\im(f)\subseteq \fm N$ if and only if $\soc_R(N^{\vee})\subseteq \ker(f^{\vee})$.
\end{lem}

\begin{proof}
Let $\pi\colon N\to N/\fm N$ be the natural surjection. Note that $\im(f)\subseteq \fm N$ if and only if $\pi f=0$ if and only if $f^{\vee}\pi^{\vee}=(\pi f)^{\vee}=0$ if and only if $\im(\pi^{\vee})\subseteq \ker(f^{\vee})$. On the other hand, we have
\begin{equation}\label{eq20260915c}
\im(\pi^{\vee})=\left(N/\fm N\right)^{\vee}\cong \soc_R(N^{\vee}).
\end{equation}
This completes the proof of the lemma.
\end{proof}

\noindent \emph{Proof of Theorem~\ref{thm20260919a}.}
(a) By~\eqref{eq20260915b} we obtain a non-trivial direct sum decomposition $\fm=J\oplus P$, where $J^{\vee}=J'$. By~\cite[Theorem A(a)]{NT}, the inclusion map $\delta\colon J\hookrightarrow R$ is an irreducible monomorphism of finitely generated $R$-modules which admits the factorization $J\hookrightarrow \fm\hookrightarrow R$ of natural inclusions whose Matlis dual identifies the diagram $E\to \overline{E}\to J'$, that is, $\delta^{\vee}=\delta'$. Thus, by Proposition~\ref{lem20260914a}, the map $\delta'$ is an irreducible epimorphism in the category of artinian $R$-modules.

(b) Let $I\cong E^{\oplus n}$ for some positive integer $n$. By Proposition~\ref{lem20260914a}, the map $\phi=\psi^{\vee}\colon N\to R^{\oplus n}$ is an irreducible monomorphism in the category of finitely generated $R$-modules, where $N=M^{\vee}$. Moreover, it follows from our assumption and Lemma~\ref{lem20260915b} that $\im(\phi)\subseteq \fm R^{\oplus n}$. Thus, by~\cite[Theorem A(b)]{NT}, the $R$-module $N$ is isomorphic to a direct summand of $\m$. Taking Matlis dual, this implies that the $R$-module $M$ is isomorphic to a direct summand of $\overline{E}$. Moreover, if $M$ is indecomposable, then $N$ is indecomposable and again, by~\cite[Theorem A(b)]{NT} we get a commutative diagram
\begin{equation}\label{eq20260915e}
\xymatrix{
N\ar[rr]^{\phi}\ar[d]_{\eta}&&R^{\oplus n}\ar[d]^{\iota^{\vee}}\\
\fm\ar[rr]^{\theta}&&R
}
\end{equation}
in which $\eta$ is a split monomorphism and  $\theta$ is the inclusion map. The existence of diagram~\eqref{eq20260915d} now follows from taking Matlis dual of diagram~\eqref{eq20260915e}.
\qed\vspace{2mm}

Here is an example that falls into the situation of Theorem~\ref{thm20260919a}.

\begin{ex}
Let $R=k[\![X,Y]\!]/(X,Y)^2$ and use lower-case letters $x,y$ to represent the residues of the variables $X,Y$ in $R$. Note that $\fm^2=0$ and thus, $\dim_k(\fm)=2$. Realizing $E$ via the isomorphism $E\cong \Hom_k(R,k)$, the basis for the $k$-vector space $E$ is the dual basis $\{1^*,x^*,y^*\}$. Note that $\soc_R(E)=k1^*$ and $\overline{E}\cong kx^*\oplus ky^*$.

Let $J'=kx^*$, which is a direct summand of $\overline{E}$, and let $\delta'\colon E\xra{\theta'}\overline{E}\twoheadrightarrow J'$ be the composition of the natural surjections, i.e., $\delta'(1^*)=0$, $\delta'(x^*)=1$, and $\delta'(y^*)=0$, where the outputs are in $k$. Indeed, $\theta'(1^*)=0$, $\theta'(x^*)=x^*$, and $\theta'(y^*)=y^*$, and the projection onto $J'$ under the identification $J'\cong k$ gives $\delta'(x^*)=1$ as a scalar in $k$. By Theorem~\ref{thm20260919a}(a), the map $\delta'$ is an irreducible epimorphism in the category of artinian $R$-modules.

To demonstrate Theorem~\ref{thm20260919a}(b), let $I=E$, $M=J'$, and $\psi=\delta'$ and note that $\soc_R(I)=k1^*\subseteq \ker(\delta')=k1^*\oplus ky^*$. As $\End_R(E)\cong R$, every monomorphism $\iota\colon E\to E$ is multiplication by a unit $r\in R$. Write $r=a+bx+cy$ with $a,b,c\in k$ and $a\neq 0$. It can be checked that $\delta'\iota=a\cdot \delta'$ and therefore, $\eta'\colon \overline{E}\to J'$ is the map defined by $\eta'(x^*)=a$ and $\eta'(y^*)=0$. Moreover, the diagram~\eqref{eq20260915d} commutes for every unit element $r\in R$.
\end{ex}

The proof of Theorem~\ref{thm20260917a} is given after the following remark.

\begin{disc}\label{disc20260917a}
In the setting of Theorem~\ref{thm20260917a}(b2), note that
$$
\Hom_R(E,k)\cong \Hom_R(E,\Hom_R(k,E))\cong \Hom_R(k,\Hom_R(E,E))\cong \Hom_R(k,R).
$$
Therefore, we get the equalities
$$
\dim_k(\Hom_R(E,k))=\dim_k(\soc(R))=1.
$$
Thus, all non-zero surjections $E\to k$ are scalar multiples of one another, that is, they share the same kernel. Hence, $N$ is the unique maximal submodule of $E$. 
\end{disc}

\noindent \emph{Proof of Theorem~\ref{thm20260917a}.}
(a) By Proposition~\ref{lem20260914a}, the map $\psi^{\vee}\colon R^{\oplus n}\to M^{\vee}$ is an irreducible epimorphism in the category of finitely generated $R$-modules. Thus, by~\cite[Theorem B(a)]{NT} we have $\ker(\psi^{\vee})\cong k$, that is, we have a short exact sequence
\begin{equation}\label{eq20260917d}
0\to k\to R^{\oplus n}\xra{\psi^{\vee}} M^{\vee}\to 0
\end{equation}
of $R$-modules. Taking Matlis dual and using the isomorphism $k^{\vee}\cong k$, we obtain a short exact sequence
$$
0\to M\xra{\psi} E^{\oplus n}\to k\to 0
$$
which implies that $\coker(\psi)\cong k$, as desired.

(b) Since $\psi^{\vee}\colon R^{\oplus n}\to M^{\vee}$ is an irreducible epimorphism in the category of finitely generated $R$-modules and $\pi^{\vee}\colon R\to R^{\oplus n}$ is a split monomorphism, the fact that $\End_R(M^{\vee})$ is a local ring along with~\cite[Theorem B(b1)]{NT} implies that the map $(\pi \psi)^{\vee}=\psi^{\vee}\pi^{\vee}$ is irreducible. Thus, by Proposition~\ref{lem20260914a}, the map $\pi \psi$ is irreducible. Moreover, if $\pi \psi$ is injective, or equivalently, $(\pi \psi)^{\vee}$ is surjective, then by~\cite[Theorem B(b2)]{NT} the ring $R$ has type $1$, $n=1$, and there is a commutative diagram
\begin{equation}\label{eq20260917b}
\xymatrix{
R^{\oplus n}\ar[rr]^{\psi^{\vee}}&&M^{\vee}\\
R\ar[rr]^-{\pi'}\ar[u]^{\beta^{\vee}}&&R/\soc(R)\ar[u]_{\rho}
}
\end{equation}
where $\rho$ is an isomorphism and $\pi'$ is the natural surjection. Taking Matlis dual of diagram~\eqref{eq20260917b} we obtain the diagram
\begin{equation}\label{eq20260917c}
\xymatrix{
M\ar[rr]^{\psi}\ar[d]_-{\rho^{\vee}}&&I\ar[d]^-{\beta}\\
(R/\soc(R))^{\vee}\ar[rr]^-{\pi'^{\vee}}&&E
}
\end{equation}
in which $\pi'^{\vee}$ is the inclusion map and $\rho^{\vee}$ is an isomorphism. On the other hand, note that $\depth(R)=0$ because of the existence of the short exact sequence~\eqref{eq20260917d}. Therefore, $\dim_k(\soc(R))=1$ because it is equal to type of $R$. Hence, $\soc(R)\cong k$ and taking Matlis dual of the short exact sequence
$$
0\to \soc(R)\to R\to R/\soc(R)\to 0
$$
we obtain a short exact sequence
$$
0\to \left(R/\soc(R)\right)^{\vee}\to E\to k\to 0.
$$
Our discussion in Remark~\ref{disc20260917a} implies that $N\cong \left(R/\soc(R)\right)^{\vee}$. Hence, the existence of diagram~\eqref{eq20260917a} follows from diagram~\eqref{eq20260917c} by setting $\varsigma=\pi'^{\vee}$ and $\alpha=\rho^{\vee}$.
\qed\vspace{2mm}

The following example falls into the situation of Theorem~\ref{thm20260917a}.

\begin{ex}
Let $R=k[\![X,Y]\!]/(X^2,Y^2)$ and use lower-case letters $x,y$ to represent the residues of the variables $X,Y$ in $R$. The ring $R$ is artinian and Gorenstein. Note that $\fm^2=kxy=\soc(R)$ and $E\cong R$. Let $I=E$ and under the isomorphism $E\cong R$ assume that $\psi\colon \fm\to E$ is the natural inclusion. By~\cite[Theorem A(a)]{NT}, the map $\psi$ is an irreducible monomorphism. Note that in this case $N=\ker(E\twoheadrightarrow k)=\fm$. Therefore, under the identification $I=E\cong R$, the maps $\psi$, $\alpha=\id_{\fm}$, and $\varsigma\colon \fm\hookrightarrow R$ fit into the diagram~\eqref{eq20260917a}.   
\end{ex}

The rest of this section is devoted to the proof of Theorem~\ref{thm20260917c}. The following result may be of independent interest. 

\begin{prop}\label{thm20260917b}
Let $\psi\colon M\to I$ be an irreducible epimorphism of $R$-modules with $I$ injective. Then, for every non-zero submodule $N$ of $\ker(\psi)$, the induced map $\psi_N\colon M/N\to I$ is a split epimorphism and we have an isomorphism
\begin{equation}\label{eq20260917f}
M/N\cong I\oplus (\ker(\psi)/N).
\end{equation}
\end{prop}

\begin{proof}
Note that $\psi$ has a factorization $M\xra{\pi} M/N\xra{\psi_N}I$, where $\pi$ is the natural surjection. Thus, either $\pi$ is a split monomorphism or $\psi_N$ is a split epimorphism. If $\pi$ is a split monomorphism, then it is an isomorphism which forces $N$ to be a zero submodule of $M$. This contradiction implies that $\psi_N$ is a split epimorphism. Moreover, we have a split short exact sequence
$0\to \ker(\psi)/N\to M/N\xra{\psi_N} I\to 0$ which implies the isomorphism~\eqref{eq20260917f}.
\end{proof}

\noindent \emph{Proof of Theorem~\ref{thm20260917c}.}
By Proposition~\ref{lem20260914a}, the map $\psi^{\vee}\colon R\to M^{\vee}$ is an irreducible monomorphism in the category of finitely generated $R$-modules. Then, by Theorem~\ref{thm20260913a}, there exists a commutative diagram
\begin{equation}\label{eq20260917g}
\xymatrix{
R\ar[rr]^{\psi^{\vee}}\ar@{=}[d]&&M^{\vee}\ar[d]^-{\cong}\\
R\ar[rr]^-{\varpi}&&R
}
\end{equation}
in which $\varpi$ is the multiplication by $x$. Note that $\varpi'=\varpi^{\vee}$. The existence of diagram~\eqref{eq20260917e} now follows from taking Matlis dual of diagram~\eqref{eq20260917g}.
\qed
\vspace{2mm}

\begin{ex}
Let $R=k[[x]]$ and $\psi: E\to E$ be multiplication by $x$. Dualizing Example~\ref{ex20260919a} (or arguing directly, using that $E$ is indecomposable and $x$ is a non-unit) shows $\psi$ is an irreducible epimorphism, matching Theorem~\ref{thm20260917c}.
\end{ex}

\subsection*{AI disclosure}
All mathematical content and research decisions in this work were directed by the author. Claude (Anthropic) was used only as a tool, under the author's direction, to draft candidate proofs and check arguments. The paper was written by the author, who takes full responsibility for all mathematical content.

\subsection*{Acknowledgments}
The author is very grateful to Akiyoshi Sannai for carefully reading the previous draft and pointing out an error in an example concerning Theorem~\ref{thm20260917c}. Correcting that example led the author to improved statements of Theorems~\ref{thm20260913a} and~\ref{thm20260917c}.

\end{document}